\documentclass[11pt]{amsart}
\usepackage{amssymb}

\newcommand{\eps}{\varepsilon}

\newcommand{\N}{{\mathbb N}}
\newcommand{\C}{{\mathbb C}}
\newcommand{\Z}{{\mathbb Z}}
\newcommand{\R}{{\mathbb R}}
\newcommand{\wt}{\widetilde}
\newcommand{\wand}{wandering domain}
\newcommand{\bwd}{Baker wandering domain}
\newcommand{\tef}{transcendental entire function}
\newcommand{\tmf}{transcendental meromorphic function}
\newcommand{\ef}{entire function}
\newcommand{\mf}{meromorphic function}

\newcommand{\sconn}{simply connected}
\newcommand{\mconn}{multiply connected}
\newcommand{\dconn}{doubly connected}
\newcommand{\econn}{eventual connectivity}
\newcommand{\fconn}{finitely connected}
\newcommand{\iconn}{infinitely connected}

\newcommand\qfor{\quad\text{for }}

\theoremstyle{plain}
\newtheorem{theorem}{Theorem}
\newtheorem{corollary}{Corollary}

\newtheorem{lemma}{Lemma}
\theoremstyle{definition}

\theoremstyle{remark}

\theoremstyle{problem}

\theoremstyle{example}
\newtheorem{example}{Example}

\begin{document}


\title[On multiply connected wandering domains of meromorphic functions]{On multiply connected wandering domains of meromorphic functions}

\author{P.J. Rippon}
\address{The Open University \\
   Department of Mathematics \\
   Walton Hall\\
   Milton Keynes MK7 6AA\\
   UK}
\email{p.j.rippon@open.ac.uk}

\author{G.M. Stallard}
\address{The Open University \\
   Department of Mathematics \\
   Walton Hall\\
   Milton Keynes MK7 6AA\\
   UK}
\email{g.m.stallard@open.ac.uk}



\subjclass{Primary: 30D05, 37F10}


\begin{abstract}
We describe conditions under which a multiply connected wandering domain of a transcendental meromorphic function with a finite number of poles must be a Baker wandering domain, and we discuss the possible eventual connectivity of Fatou components of transcendental meromorphic functions. We also show that if $f$ is meromorphic, $U$ is a bounded component of $F(f)$ and $V$ is the component of $F(f)$ such that $f(U)\subset V$, then $f$ maps each component of $\partial U$ onto a component of the boundary of $V$ in $\hat{\C}$. We give examples which show that our results are sharp; for example, we prove that a multiply connected wandering domain can map to a simply connected wandering domain, and vice versa. 
\end{abstract}

\maketitle

\section{Introduction}
\setcounter{equation}{0} Throughout this paper $f:\C\to\hat{\C}$ is a meromorphic function and we denote by $f^{n},\,n=0,1,2,\ldots\,$, the $n$th iterate of~$f$.  
The {\it Fatou set} $F(f)$ is
defined to be the set of points $z \in \C$ such that
$(f^{n})_{n \in \N}$ is well-defined, meromorphic and forms a
normal family in some neighborhood of $z$.  The complement of
$F(f)$ in $\hat{\C}$ is called the {\it Julia set} $J(f)$ of $f$. An
introduction to the properties of these sets can be found in [\ref{wB93}]. In this paper we study the components of $F(f)$, known as {\it Fatou components}, and their boundaries. Note that the notions of closure and complements are always taken with respect to $\hat{\C}$. However, we need to consider both the boundary of a set $U$ in $\C$, for which we use the notation $\partial U$, and the boundary of $U$ in $\hat{\C}$, for which we use $\hat{\partial} U$.

The set $F(f)$ is completely invariant under $f$, as is $J(f)$ in the sense that $z\in J(f)$ if and only if $f(z)\in J(f)$ whenever $f(z)$ is defined. Therefore, any component of $F(f)$ must map into a component of $F(f)$, though this mapping may not be onto because of the possible presence of finite asymptotic values; see Lemma~\ref{Herring} for more detail on this phenomenon. Similar remarks apply to components of $J(f)\cap \C$ and components of $\partial U$, where $U$ is a Fatou component; see Example~\ref{ex5}. 

For any component $U$ of
$F(f)$ there exists, for each $n=0,1,2,\ldots\,$, a component of $F(f)$, which we call $U_n$, such that $f^n(U) \subset U_n$.  If, for some $p\ge 1$, we have $U_p =U_0= U$, then we say
that $U$ is a periodic component of {\it period} $p$, assuming $p$ to be minimal. There are then five possible types of periodic components; see [\ref{wB93}, Theorem~6]. If $U_n$ is not eventually periodic, then we say that $U$ is a {\em wandering component} of $F(f)$, or a {\it wandering domain}.

We use the name {\it Baker wandering domain} to denote a wandering component $U$ of $F(f)$ such that, for $n$ large enough, $U_n$ is a bounded multiply connected component of $F(f)$ which surrounds~0, and $U_n \to \infty$ as $n\to\infty$. An example of this phenomenon with $f$ an {\ef} was first given by Baker in [\ref{iB76}] and examples with either a finite or an infinite number of poles can be obtained by minor modifications of this construction; see [\ref{RS05}].

If $f$ is a {\tef} and $U$ is a multiply connected component of $F(f)$, then $U$ is a {\bwd}; see [\ref{iB75}]. This need not be the case for meromorphic functions, even those with finitely many poles; see [\ref{Pat2}] for examples of meromorphic functions with one pole which have invariant {\mconn} components of $F(f)$. There are also examples of {\mf}s with {\mconn} {\wand}s that are not {\bwd}s. For example, in [\ref{BKL1}] Baker, Kotus and L{\" u} used techniques from approximation theory to construct several meromorphic functions, each with infinitely many poles, having multiply connected {\wand}s of various types. In particular, for $k\in \{2,3,\ldots\}$, they constructed a {\mf} with a $k$-connected bounded {\wand} which is not a {\bwd}; recall that a domain is {\it k-connected} or, equivalently, it has {\it connectivity} $k$ if $\hat{\C}\setminus U$ has $k$ components. 

Baker, Kotus and L{\" u} also showed, in [\ref{BKL3}], that any invariant Fatou component of a meromorphic function is {\sconn}, {\dconn} (in which case the component is a Herman ring) or {\iconn}. This result (apart from the Herman ring statement) was generalised by Bolsch [\ref{aB99}] to periodic Fatou components of functions that are meromorphic outside a small set of essential singularities. 

In this paper, we first study the set $M_F$ of {\tmf}s with only finitely many poles and we give conditions under which a multiply connected {\wand} of a function in $M_F$  must be a {\bwd}. We also construct examples to show that if $f\in M_F$, then a {\mconn} {\wand} of $f$ need not be a {\bwd}. For any {\mf}~$f$ we let sing\,$(f^{-1})$ denote the set of inverse function singularities of $f$, which consists of the critical values and finite asymptotic values of $f$.

In Section~2, we prove the following result. Recall that for a component $U$ of $F(f)$ and for $n=0,1,2,\ldots\,$, we denote by $U_n$ the component of $F(f)$ such that $f^n(U)\subset U_n$.

\begin{theorem}\label{BWD}
Let $f\in M_F$ and let $U$ be a {\mconn} {\wand} of $f$.
\begin{itemize}
\item[(a)]
The component $U$ is a {\bwd} if and only if infinitely many of the components $U_n,\, n=0,1,2,\ldots\,$, are multiply connected.
\item[(b)]
If
\begin{equation}
\text{sing\/}(f^{-1})\cap\bigcup_{n\ge1}U_n =\emptyset,
\end{equation}
then $U_n$ is {\mconn} for $n=0,1,2,\ldots\,$, so $U$ is a {\bwd}.
\end{itemize}
\end{theorem}

{\it Remark}\quad After submitting this paper, we learnt of the paper~[\ref{QW06}] by Qiu and Wu, which contains a result closely related to our Theorem~\ref{BWD}(a). Their hypothesis is that $U$ is wandering and all $U_n$ are {\mconn}, and they conclude that $U_n\to\infty$ as $n\to\infty$ and $U_n$ surrounds $0$ for large $n$. From this they deduce that $f$ has infinitely many weakly repelling fixed points. By Theorem~\ref{BWD}(a), this conclusion follows also from the hypothesis that $U$ is wandering and infinitely many $U_n$ are {\mconn}. 

Note that Theorem~1(a) is false without the hypothesis that $f\in M_F$. This is shown by the {\fconn} example of Baker, Kotus and L{\" u}~[\ref{BKL1}] mentioned earlier. In Section~4, we construct an {\iconn} example to show this, as follows.

\begin{example}\label{ex1}
There exists a {\mf} $f$ with infinitely many poles and a {\wand} $U$ such that each component $U_n$, $n=0,1,2,\ldots\,$, is bounded and {\iconn}, but $U$ is not a {\bwd}.
\end{example}

Our second example shows that there does exist a {\mf} $f$ with a {\mconn} {\wand} $U$ such that, for $n\ge 1$, the components $U_n$ are {\sconn}. As far as we know, this is the first such example.  

\begin{example}\label{ex2}
There exists a function $f\in M_F$ with a bounded doubly connected {\wand} $U$ such that each component $U_n,\, n=1,2,\ldots\,$, is bounded  and {\sconn}.
\end{example}

Next we discuss some general connectivity properties of Fatou components of {\tmf}s. Following Kisaka and Shishikura [\ref{KS06}], we define the {\em {\econn}} of a component~$U$ of~$F(f)$ to be~$c$ provided that $U_n$ has connectivity $c$ for all large values of $n$. Kisaka and Shishikura [\ref{KS06}, Theorem~A] showed that if $f$ is entire and $U$ is a {\mconn} component of $F(f)$, and hence a {\bwd}, then the {\econn} of $U$ exists and is either~2 or~$\infty$. Moreover, they constructed the first example of an entire function~$f$ with a {\bwd} with eventual connectivity~2, thus answering an old question; see~[\ref{BKL1}] and~[\ref{wB93}, page~167]. Earlier, Baker~[\ref{iB85}] constructed an example with infinite eventual connectivity. 

For {\mf}s the situation is less straightforward since a {\wand} can be {\mconn} without being a {\bwd}. The following theorem on connectivity properties of bounded components of $F(f)$ is a collection of known results by other authors, stated together for convenience; see Section~3 for references. Here we denote the connectivity of a domain $U$ by $c(U)$.

\begin{theorem}\label{bddU}
Let $f$ be meromorphic, let $U$ be a bounded component of $F(f)$ and let $V$ be the component of $F(f)$ such that $f(U)\subset V$.
\begin{itemize}
\item[(a)]
We have
\[f(U)=V\quad\text{and}\quad f(\partial U)=\hat{\partial} V.\]
\item[(b)]
If $U$ is {\fconn}, then $c(U)\ge c(V)$.
\item[(c)]
If $U$ is {\iconn}, then $V$ is {\iconn}.
\end{itemize}
\end{theorem}

We remark that if a pole of $f$ lies in $\partial U$, then $\partial V$ is unbounded and $\hat{\partial} V=\partial V\cup \{\infty\}$.

The following corollary of Theorem~\ref{bddU} is immediate.
\begin{corollary}\label{econn}
Let $f$ be meromorphic, let $U$ be a component of $F(f)$ and suppose that the components $U_n$, $n=0,1,2,\ldots\,$, are all bounded. 
\begin{itemize}
\item[(a)]
If $U$ is finitely connected, then 
\[c(U_n)\ge c(U_{n+1}),\qfor n=0,1,2,\ldots,\]
so the {\econn} of $U$ exists and is finite.
\item[(b)]
If $U$ is infinitely connected, then each $U_n,\, n=0,1,2,\ldots\,$, is infinitely connected, so the {\econn} of $U$ is $\infty$. 
\end{itemize}
\end{corollary}

Note that in Corollary~\ref{econn} we have $f^n(U)=U_n$, for $n\in \N$, by Theorem~\ref{bddU}(a).

Using Theorem~\ref{BWD}(a) and Corollary~\ref{econn}, we obtain the following result.  Part~(b) generalises to $M_F$ a result of Kisaka and Shishikura~[\ref{KS06}, Theorem~A] for entire functions, mentioned above.

\begin{theorem}\label{econnmf}
Let $f\in M_F$ and let $U$ be a {\wand} of $f$.
\begin{itemize}
\item[(a)]
If \;$U$ is not a {\bwd}, then the eventual connectivity of $U$ is~1.
\item[(b)]
If $U$ is a {\bwd}, then the eventual connectivity of $U$ is either~2 or~$\infty$.
\end{itemize}
\end{theorem} 
In the example of Baker, Kotus and L{\" u} mentioned after Theorem~1, it can be shown that the wandering domains have eventual connectivity $k$, where $k\in \{2,3,\ldots\}$. Thus part~(a) of Theorem~\ref{econnmf} is false without the assumption that $f\in M_F$. By modifying their example, we can obtain a {\mf} $f$ with a {\bwd} whose eventual connectivity is $k$, where $k\in \{2,3,\ldots\}$, so Theorem~\ref{econnmf}(b) is also false without the assumption that $f\in M_F$. The idea of the modification is to replace the sequence of $k$-connected domains used in the original construction, which are almost invariant under the mapping $z\mapsto z+10$, by a sequence of similarly shaped domains which are almost invariant under $z\mapsto 10z$; we omit the details which are routine but lengthy. 

We now discuss several examples related to Theorem~\ref{bddU}. First, it is well known that Theorem~\ref{bddU}(a) is false if $U$ is unbounded. For example, the function $f(z)=e^z-1$ has an unbounded immediate parabolic basin $U$, which contains the singularity $-1$, such that $f(U)= U\setminus \{-1\}$. On the other hand, for almost all $\lambda$ with $|\lambda|=1$, the function $f(z)=\lambda (e^{z}-1)$ has an unbounded invariant Siegel disc $U$, whose boundary contains the singularity $-\lambda$, such that $f(\partial U)\subset \partial U\setminus \{-\lambda\}$; see~[\ref{lR04}] and ~[\ref{pR94}].

Next we show that the requirement that $U$ is bounded is essential in Theorem~\ref{bddU}(b), as is the requirement that all $U_n$ are bounded in the statement that $c(U_n)\ge c(U_{n+1})$, for $n=0,1,2,\ldots\,$, in Corollary~\ref{econn}(a).

\begin{example}\label{ex3}
There exists a function $f\in M_F$ with a bounded {\sconn} {\wand} $U$ such that \begin{itemize}
\item[(a)]
$f(U)$ is an unbounded {\sconn} component of $F(f)$ and $\partial f(U)$ consists of two unbounded components;
\item[(b)] 
$f^2(U)$ is a bounded {\dconn} component of $F(f)$;
\item[(c)]
$f^n(U),\, n\ge 3$, are bounded {\sconn} components of $F(f)$.
\end{itemize}
Thus $U_1=f(U)$ is unbounded and $c(U_1)=1<2=c(U_2)$.
\end{example} 

The requirement that $U$ is bounded is also essential in Theorem~\ref{bddU}(c), as is the requirement that all $U_n$ are bounded in Corollary~\ref{econn}(b).

\begin{example}\label{ex4}
There exists a function $f\in M_F$ with a bounded {\iconn} {\wand} $U$ such that \begin{itemize}
\item[(a)]
$f(U)$ is an unbounded {\iconn} component of $F(f)$;
\item[(b)] 
$f^2(U)$ is contained in a bounded {\dconn} component of $F(f)$;
\item[(c)]
$f^n(U),\, n\ge 3$, are contained in bounded {\sconn} components of $F(f)$.
\end{itemize}
Thus $U_1=f(U)$ is unbounded and {\iconn}, and the eventual connectivity of~$U_1$ is~1.
\end{example} 
The following result is closely related to Theorem~\ref{bddU}. This result may also be known, but we have not been able to find a reference to it in this generality. Note that Theorem~\ref{component} gives an alternative proof of Theorem~\ref{bddU}(c).
\begin{theorem}\label{component}
Let $f$ be meromorphic, let $U$ be a bounded component of $F(f)$ and let $V$ be the component of $F(f)$ such that $f(U)\subset V$. Then $f$ maps each component of $\partial U$ onto a component of $\hat{\partial} V$.
\end{theorem}
We remark that if a pole of $f$ lies in a component of $\partial U$, then the image of that component may be the union of more than one component of $\partial V$ together with $\{\infty\}$.

Our final example shows that Theorem~\ref{component} is false if $U$ is unbounded.

\begin{example}\label{ex5}
The function $f(z)=ze^z$ has an unbounded immediate parabolic basin $U$ whose boundary $\partial U$ has components $\alpha$ and $\alpha'$ such that $f(\alpha)=\alpha'\setminus\{0\}$.
\end{example}

Finally, for an unbounded component $U$ of $F(f)$, we can obtain the following result relating the boundary connectedness properties of $U$ to those of the component of $F(f)$ which contains $f(U)$.
\begin{theorem}\label{unbddU}
Let $f$ be a {\tmf}, let $U$ be an unbounded component of $F(f)$ and let $V$ be the component of $F(f)$ such that $f(U)\subset V$. 
\begin{itemize}
\item[(a)]
We have
\[\hat{\partial} V=\overline{f(\partial U)}.\]
\item[(b)]
If $\partial U$ has only a finite number $N$ of components, then $\hat{\partial} V$ has at most $N$ components.
\item[(c)]
If $c(V)>c(U)$, then there exists at least one unbounded component of $\partial U$ which has a bounded image.
\end{itemize}
\end{theorem}
Example~3 shows that the situation in Theorem~\ref{unbddU}(c) can occur, since in this example we have $c(U_2)>c(U_1)$.

{\bf Acknowledgement}\quad We are grateful to the referee for many helpful suggestions and improvements to the paper.  

\section{Proof of Theorem~\ref{BWD}}
\setcounter{equation}{0}
First, we give several results needed in the proof of Theorem~\ref{BWD}. 

\begin{lemma}\label{largecurve}
Let $f\in M_F$. There exists $r_0>0$ such that if $U$ is a component of $F(f)$ which contains a Jordan curve surrounding  $\{z:|z|\le r_0\}$, then $U$ is a {\bwd}.
\end{lemma}
\begin{proof}
In [\ref{RS05}, Theorem~3] we proved that if $f\in M_F$, then there exists $r_0>0$ such that if $U$ is a component of $F(f)$ and $\{z:|z|\le r_0\}$ lies in a bounded complementary component of $U$, then $U$ is a {\bwd}. The proof given there depends only on the fact that $U$ contains a Jordan curve which winds round $\{z:|z|\le r_0\}$ and so it yields the above more general result.
\end{proof}

Now we denote by $M$ the set of {\tmf}s $f$ with at least one pole which is not an omitted value of $f$; in the language of [\ref{BKL2}], $f$ satisfies Assumption~A or is a `general {\mf}'. We also introduce the notation $\wt{E}$ to denote the union of a set $E$ and its bounded complementary components. 
 
\begin{lemma}\label{pole}
Let $f\in M$ and let $U$ be a component of $F(f)$. If there is a Jordan curve $\gamma$ in $U$ such that $\wt{\gamma}$ meets $J(f)$, then  for some $n\ge 0$, $\wt{f^n(\gamma)}$ contains a pole of $f$.
\end{lemma}
\begin{proof}
This follows from the fact that for $f\in M$ we have $J(f)=\overline{O^-(\infty)}$, by [\ref{BKL2}, Lemma~1], together with the fact that if $f^n$ is analytic on $\wt{\gamma}$, then $\partial f^n(\wt{\gamma})\subset f^n(\gamma)$.
\end{proof}
In the next lemma we use the classification of periodic components of $F(f)$ into five types: attracting basins, parabolic basins, Siegel discs, Herman rings and Baker domains; see~[\ref{wB93}, Theorem~6]. Here, and in the proof of Theorem~\ref{BWD}(b), we use ideas from [\ref{gS91}, Lemma~3.3].
\begin{lemma}\label{normal}
Let $f\in M\cap M_F$ and let $U$ be a component of $F(f)$. If there is a Jordan curve $\gamma$ in $U$ such that $\wt{f^n(\gamma)}$ contains a point of $J(f)$ for infinitely many $n$, then $U$ is either a Herman ring (or its pre-image) or a {\bwd}.
\end{lemma}
\begin{proof}
Suppose that $U$ is not a Herman ring (nor its pre-image).  Clearly $U$ is not a Siegel disc (nor its pre-image). Therefore $U$ is a {\wand} or an immediate attracting or parabolic basin of $F(f)$, or a Baker domain of~$f$ (or a pre-image of one of these). Hence all locally uniformly convergent subsequences of $f^n$ have constant limit functions in~$U$; see~[\ref{BKL3}, Lemma~2.1] and~[\ref{wB93},~page 163]. Thus the spherical diameter of $\gamma_n=f^n(\gamma)$ tends to~0 along any such subsequence. Since $f\in M$ and $f\in M_F$, we deduce by Lemma~\ref{pole} that, for infinitely many $n$, $\wt{\gamma_n}$ contains the same pole of~$f$, say~ $p$. Thus there is a sequence $n_k$ such that $p\in\wt{\gamma_{n_k}}$ for all $k$ and $f^{n_k}$ tends to either~$\infty$ or~$p$, locally uniformly in $U$. 

In the first case, dist$(\gamma_{n_k},0)\to\infty$. Also, $p\in\wt{\gamma_{n_k}}$ and hence $0\in\wt{\gamma_{n_k}}$, for all large enough~$k$. Thus $U$ is a {\bwd} by Lemma~\ref{largecurve}. In the second case, dist$(\gamma_{n_k},p)\to 0$, so dist($f(\gamma_{n_k}),0)\to\infty$ and $0\in\wt{f(\gamma_{n_k}})$, for all large enough $k$. Thus $U$ is again a {\bwd} by Lemma~\ref{largecurve}.
\end{proof}

\begin{proof}[Proof of Theorem~\ref{BWD}(a)]
First, if $f$ is a {\tef}, then Theorem~\ref{BWD}(a) is well-known; see~[\ref{iB75}]. Next, suppose that $f$ is a {\tmf} with exactly one pole, which is an omitted value of $f$. Then $f$ cannot have a {\mconn} {\wand}~[\ref{iB87}, Theorem~1], so there is nothing to prove. Hence we can assume without loss of generality that $f\in M\cap M_F$.

It is obvious that if $U$ is a {\bwd}, then infinitely many $U_n$ are {\mconn}. We now prove the opposite implication by contradiction. Let $U$ be a {\wand} such that infinitely many of the components $U_n$ are {\mconn} and suppose that $U$ is not a {\bwd}. Since $U$ is a {\wand}, we deduce, by Lemma~\ref{normal}, that 
\begin{quote}
if $\gamma$ is a Jordan curve in $U_N$, where $N\ge 0$, then $\wt{f^n(\gamma)}$ contains a pole of $f$ for at most finitely many $n$.
\end{quote}
Choose $n_0$ such that $U_{n_0}$ is {\mconn}, and then take any Jordan curve $\gamma_0$ in $U_{n_0}$ such that $\wt{\gamma_0}$ meets $J(f)$. By Lemma~\ref{pole}, we can choose $m_0\ge 0$ such that $\wt{f^{m_0}(\gamma_0)}$ contains a pole of $f$.  If $\wt{f^{m_0+1}(\gamma_0)}$ meets $J(f)$, then we can apply Lemma~\ref{pole} again to find $m'_0>m_0$ such that $\wt{f^{m'_0}(\gamma_0)}$ contains a pole of $f$. Repeating this argument as often as necessary we deduce, by the above displayed statement, that we can redefine~$m_0$ to be a non-negative integer such that $\wt{f^{m_0}(\gamma_0)}$ contains a pole of $f$ and $\wt{f^{m_0+1}(\gamma_0)}$ does not meet $J(f)$.

Since infinitely many of the components $U_n$ are {\mconn}, we can now choose $n_1\ge n_0+m_0+1$ and take a Jordan curve $\gamma_1$ in $U_{n_1}$ such that $\wt{\gamma_1}$ meets $J(f)$. By the above reasoning, there exists $m_1\ge 0$ such that $\wt{f^{m_1}(\gamma_1)}$ contains a pole of~$f$ and $\wt{f^{m_1+1}(\gamma_1)}$ does not meet $J(f)$. Repeating this argument, we obtain sequences of non-negative integers $n_k,\;m_k$, and Jordan curves $\gamma_k$, such that, for $k\ge 0$,
\begin{equation}
n_{k+1}\ge n_k+m_k+1,
\end{equation}
\begin{equation}
\gamma_k\subset U_{n_k}\text{ and } \wt{\gamma_k} \text{ meets } J(f),
\end{equation}
\begin{equation}
f^{m_k}(\gamma_k)\subset U_{n_k+m_k}\text{ and } \wt{f^{m_k}(\gamma_k)} \text{ contains a pole of } f,
\end{equation}
\begin{equation}
f^{m_k+1}(\gamma_k)\subset U_{n_k+m_k+1}\text{ and } \wt{f^{m_k+1}(\gamma_k)} \text{ does not meet } J(f).
\end{equation}
Since $f\in M_F$, we can assume by~(2.3) and~(2.4) that $n_k$ and $m_k$ have been chosen such that, for some pole~$p$ of $f$, 
\begin{equation}
U_{n_k+m_k}\text{ contains a Jordan curve }\Gamma_k \text{ such that } p\in \wt{\Gamma_k},
\end{equation}
\begin{equation}
\wt{f(\Gamma_k)} \text{ does not meet } J(f).
\end{equation}
Since $U$ is a {\wand}, the components $U_n$ are disjoint. Thus, for $k\ge 0$, the Jordan curves $\Gamma_k$ are disjoint by~(2.1) and~(2.5), as are the image curves $f(\Gamma_k)$. Hence, for $0\le k<l<\infty$, we must have $\Gamma_k$ inside $\Gamma_l$, or vice versa. Since $f\in M_F$, there must exist integers $k_1$ and $k_2$, $0\le k_1<k_2<\infty$, such that $f$ has no poles in the closure of the ring domain $A$ lying between $\Gamma_{k_1}$ and $\Gamma_{k_2}$. Thus $f(A)$ is bounded and
\[\partial f(A)\subset f(\partial A)=f(\Gamma_{k_1})\cup f(\Gamma_{k_2}),\] 
so $f(A)$ is a subset of at least one of $\wt{f(\Gamma_{k_1})}$, $\wt{f(\Gamma_{k_2})}$. 
This contradicts~(2.6), however, because $A\cap J(f)\ne \emptyset$ (since $\Gamma_{k_1}$ and $\Gamma_{k_2}$ lie in different components of $F(f)$) so $f(A)\cap J(f)\ne \emptyset$. This completes the proof of Theorem~\ref{BWD}(a).

{\it Proof of Theorem~\ref{BWD}(b)}\quad Part~(b) now follows from part~(a) by a standard argument which we give for completeness.  Suppose that \begin{equation}\label{sing}
\text{sing\/}(f^{-1})\cap\bigcup_{n\ge1}U_n =\emptyset.
\end{equation} 
By part~(a), it is sufficient to prove that if $\gamma$ is any Jordan curve in $U$ which is not null-homotopic, then the image $\gamma_n=f^n(\gamma)$ is not null-homotopic in $U_n$, for $n\in \N$. But if $z_0\in \gamma$ and $\gamma_n\sim f^n(z_0)$ in $U_n$, for some $n\ge 1$, then the branch, $g$ say, of $f^{-n}$ such that $g(f^n(z_0))=z_0$ can be continued analytically (and univalently) to a simply-connected neighbourhood of $\gamma_n$ in $U_n$, by~(2.7). Then $g$ lifts the homotopy $\gamma_n\sim f^n(z_0)$ in $U_n$ to a homotopy $\gamma\sim z_0$ in $U$, which is a contradiction. This completes the proof of Theorem~\ref{BWD}(b).
\end{proof}

\section{Proofs of Theorems~2,~3,~4 and~5}
\setcounter{equation}{0}
Theorem~\ref{bddU} is a combination of the following two known results which together show that a meromorphic function $f$ maps bounded components of $F(f)$ in a nice way. An analytic function defined on a domain $U$ is called a {\em proper} map if $f$ has a topological degree; see [\ref{nS93}, pages~4--9] for a discussion of proper maps. 

\begin{lemma}\label{RiemannH}
Let $f$ be meromorphic and let $U$ be a bounded domain in which $f$ is analytic. 
\begin{itemize}
\item[(a)]
Then $f: U \to f(U)$ is proper if and only if $\hat{\partial} f(U)=f(\partial U)$ or, equivalently, if and only if pre-images of relatively compact subsets of $f(U)$ are relatively compact subsets of $U$.
\item[(b)]
If $f: U \to f(U)$ is proper with degree $k$ and  there are $N$ critical points of $f$ in $U$ (counted according to multiplicity), then 
\[c(U)-2=k(c(f(U))-2)+N;\]
in particular, $c(U)\ge c(f(U))$.
\end{itemize}
\end{lemma}
Lemma~\ref{RiemannH}(a) is proved in [\ref{nS93}, page~5, Theorem~1] and Lemma~\ref{RiemannH}(b) is the Riemann--Hurwitz formula; see [\ref{nS93}, page~7] for the case of finite connectivity and [\ref{aB99}, Lemma~4] for the case of infinite connectivity,  in an even more general context.

\begin{lemma}\label{Herring}
Let $f$ be meromorphic and let $f:U\to V$, where $U$ and $V$ are components of $F(f)$.
\begin{itemize}
\item[(a)]
Then $|V\setminus f(U)|\le 2$ and for any $w_0\in V\setminus f(U)$ there exists a path $\Gamma\subset U$ such that $f(z)\to w_0$ as $z\to\infty,\, z\in\Gamma$.
\item[(b)]
If $U$ is also bounded, then $f(U)=V$ and $f(\partial U)=\hat{\partial} V$.
\end{itemize}
\end{lemma}
Lemma~\ref{Herring}(a) and the first assertion of Lemma~\ref{Herring}(b) are results of Herring [\ref{mH98}, Theorems~1 and~2]; see also [\ref{aB99}]. Also, if $U$ is a bounded Fatou component, then it is well-known that $f:U\to V$ is proper; that is, $f(\partial U)=\hat{\partial}f(U)=\hat{\partial} V$.

All parts of Theorem~\ref{bddU} follow immediately from Lemmas~\ref{RiemannH} and~\ref{Herring}.

\begin{proof}[Proof of Theorem~\ref{econnmf}]
The proof follows that of~[\ref{KS06}, Theorem~A]. Let~$f\in M_F$ and suppose that $U$ is a {\wand}. If $U$ is not a {\bwd}, then by Theorem~\ref{BWD}(a) all but a finite number of the components $U_n$ are {\sconn}, so the {\econn} of~$U$ is~1. If $U$ is a {\bwd} which is {\iconn}, then its {\econn} is~$\infty$ by Corollary~\ref{econn}(b). If $U$ is a {\bwd} which is {\fconn}, then the {\econn}, $c$ say, of $U$ exists, by Corollary~\ref{econn}(a), and $2\le c<\infty$. If $c>2$, then $f:U_n\to U_{n+1}$ is univalent, for large~$n$, by Lemma~\ref{RiemannH}(b). Moreover, for $n$ large enough $f$ maps the outer boundary of $U_n$ to the outer boundary of $U_{n+1}$; see [\ref{Pat2}, proof of Theorem~F] or [\ref{RS05}, Lemma~4]. Thus, since $f\in M_F$, we can use the argument principle to show that $f$ takes each value in $\C$ at most finitely often, and this is impossible by Picard's theorem. Hence $c=2$, as required.
\end{proof}

\begin{proof}[Proof of Theorem~\ref{component}] 
For the case when $U$ is of finite connectivity, see [\ref{nS93}, page~6], and also [\ref{cMwR91}] for the case when in addition $U=V$. 

Let $\alpha$ be any component of $\partial U$ which is mapped into but not onto a component $\beta$ of $\hat{\partial} V$. Choose a point $w_0\in \beta\setminus f(\alpha)$, possibly $w_0=\infty$. Since $U$ is bounded and $f$ is meromorphic, there exist only finitely many pre-images of $w_0$ in $\partial U$, say $z_k$, $k=1,\ldots,p$, none of which lies in~$\alpha$.

Let $V_n$, $n=1,2,\ldots$, be a smooth exhaustion of $V$; that is, the sets $V_n$ are smooth bounded domains such that $\overline{V_n}\subset V_{n+1}$, for $n=1,2,\ldots$ and $\bigcup V_n = V$. Then $\beta$ lies in a unique component of the complement of $V_n$, for each $n$, so there exists a unique component, $H_n$ say, of $V\setminus \overline{V_n}$ such that $\beta\subset \overline{H_n}$. Note that $\beta\subset \overline{H_{n+1}}\subset \overline{H_n}$, for $n=1,2,\ldots$, so $\bigcap \overline{H_n}$ is a connected subset of $\hat{\partial} V$ and hence $\bigcap \overline{H_n}=\beta$.

We now wish to choose, for each $n$, a component $G_n$ of $U\cap f^{-1}(H_n)$ such that $\alpha\subset \overline{G_n}$. In order to do this, we construct a path $\Gamma:\,\gamma(t),\; t\in[0,\infty)$, in $U$ which approaches $\alpha$ in the sense that ${\rm dist}_{\chi}(\gamma(t),\alpha)\to 0$ as $t\to\infty$ and $\alpha\subset\overline{\Gamma}$, where~$\chi$ denotes the spherical metric on $\hat{\C}$. Such a path $\Gamma$ can be constructed by using a smooth exhaustion $U_m$ of $U$ and choosing $\Gamma$ to lie eventually outside each $U_m$ and to accumulate at each point of a dense subset of $\alpha$. Then ${\rm dist}_{\chi}(f(\gamma(t)),\beta)\to 0$ as $t\to\infty$. Thus, for each $n=1,2,\dots$, we have $f(\gamma(t))\in H_n$ for $t$ large enough, so we can define $G_n$ to be the component of $U\cap f^{-1}(H_n)$ such that $\gamma(t)\in G_n$ for $t$ large enough. By the properties of $H_n$ and the fact that $\alpha\subset \overline{\Gamma}$, we have $\alpha\subset \overline{G_{n+1}}\subset \overline{G_n}$, for $n=1,2,\ldots$. Thus $\bigcap\overline{G_n}$ is connected, contains $\alpha$, and is a subset of $\partial U$ (because any point in $\bigcap\overline{G_n}$ must be mapped by $f$ to a point in $\beta$). Hence $\bigcap \overline{G_n}=\alpha$, so we can choose $n$ such that $\overline{G_n}\cap \bigcup_{k=1}^p \{z_k\}=\emptyset$. 

For such a choice of $n$, let $w_m$ be a sequence in $H_n$ which converges to $w_0$. Since $f:G_n\to H_n$ is proper, there exists a sequence $z_m$ in $G_n$ such that $f(z_m)=w_m$, for $m=1,2,\ldots$, and we may assume that $z_m\to z_0$, where $f(z_0)=w_0$. Then $z_0\in\overline{G_n}$, a contradiction to the above choice of~$n$.
\end{proof}

To prove Theorem~\ref{unbddU}, we need some ideas from the theory of cluster sets.
First, for an unbounded domain $U$, with $z_0\in \hat{\partial} U$, we define the cluster sets
\[C_U(f,z_0)= \{w_0\in \hat{\C}:\exists \;z_n\in U \text{ with }z_n\to z_0,\,f(z_n)\to w_0\}\] 
and 
\[C_{\partial U}(f,\infty)= \{w_0\in \hat{\C}:\exists \;z_n\in\partial U \text{ with }z_n\to \infty,\,f(z_n)\to w_0\},\]
where we assume that $\partial U$ is unbounded.

We shall use the following result, which is a special case of the Beurling--Kunugui theorem; see [\ref{kN60}, page~23, Theorem~7].
\begin{lemma}\label{BeurKun}
Let $f$ be meromorphic and let $U$ be an unbounded domain such that $\partial U$ is unbounded. Suppose that the set 
\[\Omega=C_U(f,\infty)\setminus C_{\partial U}(f,\infty)\]
is non-empty and $\Omega'$ is any component of $\Omega$. Then every value from $\Omega'$, with at most two exceptions, is assumed by $f$ infinitely often in $U\cap\{z:|z|>R\}$, for all $R>0$.
\end{lemma}
The set $\Omega$ defined in Lemma~\ref{BeurKun} is open (see [\ref{kN60}, page~17, Theorem~4]) and hence $\Omega$ has at most countably many such components $\Omega'$. In particular, in Lemma~\ref{BeurKun} the set $\Omega\setminus f(U)$ is at most countable.

In the general Beurling--Kunugui theorem, the function $f$ is assumed to be meromorphic only in $U$, so $f$ need not have a continuous extension to $\partial U$ (as is the case here), and the cluster set $C_{\partial U}(f,\infty)$ is defined in terms of the values of $C_U(f,z)$, for $z\in \partial U$.

\begin{proof}[Proof of Theorem~\ref{unbddU}]
Let $f$ be a {\tmf} and let $U$ be an unbounded component of $F(f)$. Then $\partial U$ is unbounded, since $J(f)$ is unbounded, so Lemma~\ref{BeurKun} can be applied. It is a straightforward matter to check that
\begin{equation}
\hat{\partial} f(U)=f(\partial U)\cup \left(C_U(f,\infty)\setminus f(U)\right).
\end{equation}
Thus, by Lemma~\ref{Herring}(a),
\begin{equation}\label{boundary}
\hat{\partial} V=f(\partial U)\cup \left(C_U(f,\infty)\setminus (f(U)\cup E)\right),
\end{equation}
where $V$ is the component of $F(f)$ such that $f(U)\subset V$ and $E=V\setminus f(U)$, $|E|\le 2$. Note that $f(\partial U)\cap E=\emptyset$, since there are no isolated points of $\partial U$. Since $f(\partial U)\subset \hat{\partial} V$, we deduce that 
\[\overline{f(\partial U)}\subset \hat{\partial} V.\]
To prove the desired statement that $\overline{f(\partial U)}= \hat{\partial} V$, we suppose that there exists  $w_0\in \hat{\partial} V\setminus\overline{f(\partial U)}$. Then there is an open disc $\Delta$ in $\hat{\C}$ with centre $w_0$ such that 
$\Delta\cap \overline{f(\partial U)}=\emptyset$. Since $\hat{\partial} V$ is perfect, as can easily be checked by using the fact that $J(f)$ is perfect, the disc $\Delta$ contains uncountably many points $w$ such that $w\in \hat{\partial} V\setminus\overline{f(\partial U)}$. Therefore, by~(\ref{boundary}), the set 
\[\hat{\partial} V\setminus\overline{f(\partial U)}=C_U(f,\infty)\setminus \left(f(U)\cup E\cup\overline{f(\partial U)}\right)\]
is uncountable. Since $|E|\le 2$ and $C_{\partial U}(f,\infty)\subset \overline{f(\partial U)}$, the set \[C_U(f,\infty)\setminus \left(f(U)\cup C_{\partial U}(f,\infty)\right)=\Omega\setminus f(U)\]
is also uncountable, which contradicts the statement following Lemma~\ref{BeurKun}. This completes the proof of Theorem~\ref{unbddU}(a).

The proof of part~(b) is clear since $\hat{\partial}V=\overline{f(\partial U)}$, by part~(a), and $f(\partial U)$ can have at most $N$ components. 

To prove part~(c), we suppose that $c(V)>c(U)$. Then $U$ must have a finite number of bounded boundary components, $\alpha_1,\ldots,\alpha_m$ say, and there must exist at least one bounded boundary component, $\beta_0$ say, of $V$ which does not contain any of $f(\alpha_1),\ldots,f(\alpha_m)$. Let $\beta_1, \ldots,\beta_n$ denote those bounded boundary components of $V$ which contain at least one of the sets $f(\alpha_1),\ldots,f(\alpha_m)$; clearly $n\le m$.

Now suppose that $\beta_0$ is not the outer boundary of $V$. Let $\Gamma$ be a Jordan curve in $V$ which separates $\beta_0$ from $\beta_1\cup\cdots\cup \beta_n$, such that $\beta_0$ lies in the bounded complementary component, $G$ say, of $\Gamma$. This is possible by repeated applications of the result~[\ref{New}, page 143, Theorem~3.3] to the closed set $\hat{\C}\setminus V$. By part~(a), we have $f(\partial U)\cap G\ne \emptyset$. However, $f(\partial U)\cap \Gamma=\emptyset$, since $f(\partial U)\subset J(f)$. Thus if we choose $z_0\in\partial U$ such that $f(z_0)\in G$, then the component $E_0$ of $\partial U$ which contains $z_0$ is unbounded but its image lies entirely inside $\Gamma$ and so is bounded, as required.

In the case when $\beta_0$ is the outer boundary of $V$ (which can only occur when $V$ is bounded), a similar argument applies, except that in this case $\beta_0$ lies in the unbounded complementary component of $\Gamma$ and the image of $E_0$ is bounded because it lies in $\overline{V}$. This completes the proof of Theorem~\ref{unbddU}.
\end{proof}

\section{Examples}
\setcounter{equation}{0}
\setcounter{example}{0}
Our first example shows that Theorem~1(a) is false without the hypothesis that $f\in M_F$.
\begin{example}
There exists a {\mf} $f$ with infinitely many poles and a {\wand} $U$ such that each component $U_n$, $n=0,1,2,\ldots\,$, is bounded and {\iconn}, but $U$ is not a {\bwd}.
\end{example}
\begin{proof}
The construction of Example~1 is based on the entire function
\[
h(z)=2+2z-2e^z,
\]
which is derived from Bergweiler's example $z\mapsto 2-\ln 2+2z-e^z$ in [\ref{wB95}] by shifting the super-attracting fixed point from $\ln 2$ to 0. Here we consider the closely related meromorphic function
\[
f(z)=2+2z-2e^z+\frac{\eps}{e^z-e^a},
\]
where $a$ and $\eps$ are positive constants to be chosen suitably small. Note that \[\phi(z)=f(z)-2z\] 
is $2\pi i$-periodic.

First we claim that if $0<a<1/32$ and $0<\eps\le a^2/16$, then the set
\[\Delta_a=\{z:|z|\le 2a, |z-a|\ge a/2\}\]
is mapped by $f$ into $\{z:|z|<a/2\}\subset \Delta_a$. For $|z|\le 1$ we have
\begin{equation}
|2+2z-2e^z|=|z^2+z^3/3+\cdots|\le |z|^2(1+|z|/3+|z|^2/3^2+\cdots)< 2|z|^2.
\end{equation}
Similarly, $|e^z-1|\ge\frac12|z|$, for $|z|\le\frac12$, so
\begin{equation}
\left|\frac{\eps}{e^z-e^a}\right|=\frac{\eps}{e^a|e^{z-a}-1|}\le \frac{4\eps}{a}\le\frac a4\,,
\qfor a/2\le|z-a|\le1/2.
\end{equation}
The estimates~(4.1) and~(4.2) give
\[|f(z)|<8a^2+\frac a4<\frac a2\,,\qfor z\in\Delta_a,\]
since $0<a<1/32$. Therefore $f(\Delta_a)\subset \{z:|z|<a/2\}\subset \Delta_a$, as required.

Thus $f$ has a fixed point, $z_0$ say, in the interior of $\Delta_a$, which must be attracting. The corresponding immediate attracting basin $U_0$ of $f$ contains $\Delta_a$ but not the point $a$, where $f$ has a pole, so $U_0$ is multiply connected. Hence $U_0$ must be infinitely connected by [\ref{BKL3}, Theorem~3.1].

It is shown in [\ref{K98}, proof of Theorem~4] that the immediate super-attracting basin of $h$ which contains the super-attracting fixed point~$0$ is bounded. This is done by specifying a Jordan curve $\Gamma$ which winds round~0 (and is contained in $\{z:|\Im(z)|<\pi\}$), such that $h(\Gamma)$ lies in the unbounded component of the complement of $\Gamma$. This property remains true for $f(\Gamma)$ as long as we choose $\eps$ small enough and hence $U_0$ is bounded.

Since $f(z)=2z+\phi(z)$, where $\phi$ is $2\pi i$-periodic, the set $J(f)$ is $2\pi i$-periodic; see [\ref{RS00}, Corollary~1], for example. Thus, for each $n\in\Z$, the set $U_n=U_0+2n\pi i$ is a bounded {\iconn} component of $F(f)$. Now, for $n\in \Z$, we have 
\[2n\pi i\in\Delta_a+2n\pi i\subset U_n,\quad f(2n\pi i)=4n\pi i+\frac{\eps}{1-e^a}\quad\text{and}\quad
\left|\frac{\eps}{1-e^a}\right|\le\frac{a^2/16}{a}<\frac a2,\]
so $f(U_n)\subset U_{2n},$ for $n\in \Z$. Thus $U_1$ is a bounded infinitely connected {\wand} which is not a {\bwd}, as required.
\end{proof}
Note that in this example the Fatou components which contain $f^n(U_1)$ are all {\iconn}, as expected by Corollary~\ref{econn}(b).

A similar construction to Example~1 can be carried out starting with 
\[h(z)=z-1+e^{-z}+2\pi i.\]
The function $z\mapsto z-1+e^{-z}$ has congruent super-attracting basins containing the super-attracting fixed points $2n\pi i,\,n\in\Z$, and it was shown by Herman that these components form an orbit of wandering domains of $h$; see [\ref{mH84}]. In this case, the construction in Example~1 gives a {\mf} with an orbit of unbounded infinitely connected {\wand}s. We omit the details.

Our next example shows that there does exist a {\mf} with a {\mconn} {\wand} $U$ such that $U_n$ is {\sconn} for $n\ge 1$.
\begin{example}
There exists a function $f\in M_F$ with a bounded doubly connected {\wand} $U$ such that each component $U_n,\, n=1,2,\ldots\,$, is bounded  and {\sconn}.
\end{example}
\begin{proof}
The construction of Example~2 is based on the entire function
\[
g(z)=z+\lambda \sin(z+a),
\]
where $\lambda>0$ and $a\in\R$ are chosen so that $g(2n\pi)=(2n+2)\pi,\,n\in\Z$, and $g$ has critical points at each $2n\pi,\,n\in\Z$. Thus 
\begin{equation}\lambda\sin a =2\pi,\quad 1+\lambda\cos a=0,
\end{equation} 
so $a=\pi-\tan^{-1}(2\pi)=1.728\ldots\,$ and $\lambda=\sqrt{1+4\pi^2}=6.362\ldots\,$. Devaney showed in [\ref{rD89}] that $g$ has a wandering domain containing~0. Here we consider the closely related function
\[f(z)=g(z)+\frac{\eps}{z}=z+\frac{\eps}{z}+\lambda \sin(z+a),\]
where $\eps$ is a positive constant to be chosen suitably small. In particular, we require that $0<\eps<1/2$, which implies by a calculation that 
\[f(\pi/2-a)=\pi/2-a+\frac{\eps}{\pi/2-a}+\lambda>0,\]
so $f$ has a zero in the interval $(\pi/2-a,0)$. Thus $f\in M$, since $0$ is a pole of $f$.

We write $B(z,r)=\{w:|w-z|<r\},\,r>0$. Since $g$ has critical points at $2n\pi,\,n\in\Z$, and $g(z+2\pi)=g(z)+2\pi$, we can choose a constant $r_1$ such that $0<r_1<1/2$ and
\begin{equation}
|g'(z)|\le\frac14,\qfor |z-2n\pi|\le r_1,\, n\in\Z.
\end{equation}
Hence
\[g(B(2n\pi,r))\subset B((2n+2)\pi,r/4),\quad\text{for }0<r\le r_1,\,n\in\Z.\] 
(See~(4.8) for a more precise estimate of the behaviour of $g$ near~0.) Therefore, we can choose $\eps>0$ and $r_2,\,0<r_2<r_1$, such that $6\sqrt{\eps}<r_1$ and 
\begin{equation}
\overline{f(B(2n\pi,r_1))}\subset B((2n+2)\pi,r_2),\qfor n\ge1.
\end{equation}
In particular, note that $0<\eps<(r_1/6)^2<1/144$.

Now let
\[\Delta_0=\{z:\sqrt{\eps}/2<|z|<2\sqrt{\eps}\}\quad\text{and}\quad \Delta_n=B(2n\pi, r_1),\;n\ge 1.\]
The function $z\mapsto z+\eps/z$ is a Joukowski function which maps $\Delta_0$  
in a 2-to-1 manner onto an ellipse contained in $B(0,3\sqrt{\eps}\,)$. Also, by~(4.4) with $n=0$, we have
\begin{align}|\lambda\sin(z+a)-2\pi|&=|g(z)-2\pi-z|\notag\\ 
&\le|g(z)-2\pi|+|z|\notag\\ 
&\le\tfrac12 \sqrt{\eps}+2\sqrt{\eps}<3\sqrt{\eps},
\quad\text{for }z\in \Delta_0.\notag
\end{align}
Hence
\begin{equation}
f(\Delta_0)\subset B(2\pi,3\sqrt{\eps}+3\sqrt{\eps}\,)\subset B(2\pi,r_1).
\end{equation}
Therefore, by~(4.5) and (4.6),
\begin{equation}
f^n(\Delta_m)\subset \Delta_{m+n},\qfor m,n\ge 0,
\end{equation}
so
\[\Delta_0\cup\Delta_1\cup \Delta_2\cup\cdots\subset F(f),\]
by Montel's theorem. For $n\ge 0$, let $U_n$ be the component of $F(f)$ which contains $\Delta_n$. Clearly $U_0$ is multiply connected, since $0\in J(f)$, and $f^n\to\infty$ locally uniformly in each $U_n,\, n\ge 0$, by~(4.7). Hence $U_0$ is not a Herman ring (nor its pre-image). Also note that $J(f)$ is symmetric with respect to the real axis and each interval of the form $[(2n+1)\pi,(2n+2)\pi],\,n\ge 0$, contains a repelling fixed point of $f$, since $0<\eps<1/144$.

We now show that the components $U_n,\,n\ge 0$, are all different. Suppose, for a contradiction, that $U_p=U_q$, where $0\le p<q$. Then there is a Jordan curve $\gamma$ in $U_p$, which is symmetric with respect to the real axis and passes through $\Delta_p$ and $\Delta_q$. Hence $f^n(\gamma),\,n\ge0$, is a closed curve in $F(f)$, symmetric with respect to the real axis, which passes through $\Delta_{p+n}$ and $\Delta_{q+n}$. It follows that, for $n\ge 0$, the set $\wt{f^n(\gamma)}$ contains the repelling fixed point of $f$ located in the interval $[(2(p+n)+1)\pi,(2(p+n)+2)\pi]$. Thus $U_0$ is a {\bwd}, by Lemma~\ref{normal}. Therefore 
\[\frac{\ln \ln |f^n(z)|}{n}\to \infty,\quad\text{for } z\in U_0,\]
by [\ref{RS05}, Theorem~1(d)], and this contradicts the fact that $f^n(\Delta_0)\subset \Delta_n$, for $n\ge 0$. Hence the components $U_n$ are indeed different and so $U_0$ is a {\wand} but not a {\bwd}. 

We now show that the components $U_n$ are all bounded. For $n\ge 0$, put 
\[C_n=\{z:|z-2n\pi|=0.5\}\quad\text{and}\quad C'_n=\{z:|z-2n\pi|=0.6\}.\]

\begin{lemma}\label{imagecircle}
We can choose $\eps>0$ so small that, for $n\ge 0$, we have
\begin{itemize}
\item[(a)]
$f(C_n)$ winds twice positively round $C'_{n+1};$
\item[(b)]
$f'(C_n)$ winds once positively round~$\{z:|z|=1\}$;
\item[(c)]
$U_n$ lies inside $C_n$.
\end{itemize}
\end{lemma}
\begin{proof}
Recall that $g(z)=z+\lambda \sin(z+a)$ and $f(z)=g(z)+\eps/z$. In view of~(4.3), we have 
\begin{equation}
g(z)=z-\sin z+2\pi\cos z=2\pi-\pi z^2\left(1-\frac{z}{3!\pi}-\frac{2z^2}{4!}+\cdots\right).
\end{equation}
Part~(a) now follows immediately from the estimate
\begin{equation}
\left|-\frac{z}{3!\pi}-\frac{2z^2}{4!}+\cdots\right|<0.1,\quad\text{for }|z|\le 0.5,
\end{equation}
and the facts that $g(z+2\pi)=g(z)+2\pi$ and $0<\eps<1/144$. Part~(b) follows by a similar argument with 
\[g'(z)=-2\pi z\left(1-\frac{z}{2!2\pi}-\frac{z^2}{3!}+\cdots\right).\]

To prove part~(c), we first show that, for each $N\ge 0$, the family 
\[\phi_n(z)=f^n(z)-2(n+N)\pi,\quad n\ge 0,\] 
is normal in $U_N$. This holds because the components $U_n,\,n\ge 0$, are disjoint, so $f^n(z)\ne 2m\pi$, for $m>n+N,\,z\in U_N$, and hence each function $\phi_n$ omits in $U_N$ the three values
\[
\infty,\quad 2(n+1+N)\pi-2(n+N)\pi= 2\pi\quad\text{and}\quad 2(n+2+N)\pi-2(n+N)\pi= 4\pi.
\]
Using~(4.4) and making a smaller choice of $\eps$ if necessary, we deduce that 
\[|f'(z)|\le c,\qfor |z-2n\pi|\le r_1,\, n\ge 1,\]
for some $c,\,0<c<1$. Thus $f$ is contracting on each disc $\Delta_n,\,n\ge 1$. By~(4.7), for each $N\ge0$, we have diam $f^n(\Delta_N)\to 0$ as $n\to\infty$, so there exists $a_N$ with $|a_N|\le r_1<1/2$ and a subsequence $n_k$ such that
\begin{equation}
\phi_{n_k}(z)\to a_N\;\text{ as }k\to\infty,\quad\text{locally uniformly in }U_N.
\end{equation}
 
Now suppose for a contradiction that $U_N\cap C_N\ne\emptyset$, for some $N\ge 0$. Then we can join a point $z_N$ of $\Delta_N$ to a point $w_N\in C_N$ by a compact curve $\Gamma$ lying in $U_N$. Since $f^n(z_N)\in\Delta_{n+N}$ for all $n>0$, we deduce that $f^n(\Gamma)$ meets $C_{n+N}$ and $C'_{n+N}$ for all $n>0$. This contradicts~(4.10) and completes the proof of Lemma~\ref{imagecircle}. 
\end{proof}

We now continue the proof of Example~\ref{ex2}. Since the components $U_n$ are all bounded, we deduce that $U_n=f^n(U_0),\,n\ge 0$, by Lemma~\ref{Herring}(b).

We can now deduce that the components $U_n,\,n\ge1$, are all simply connected. Indeed, if $N\ge 1$ and $\gamma_N$ is a Jordan curve in $U_N$ which is not null-homotopic in $U_N$, then for some $n\ge 0$ the set $\wt{f^n(\gamma_N)}$ must contain a pole of $f$, by Lemma~\ref{pole}, and this is impossible by Lemma~\ref{imagecircle}(c).

Finally, we show that $U_0$ is doubly connected. To do this we use the Riemann--Hurwitz formula
\begin{equation}
c(U_0)-2=k_0(c(U_1)-2)+N_0,
\end{equation}
where $k_0$ is the degree of the (proper) mapping $f:U_0\to U_1$ and $N_0$ is the number of critical points of $f$ in $U_0$; see Lemma~\ref{RiemannH}(b).

By Lemma~\ref{imagecircle}(a), with $n=0$, and the argument principle, the set $\{z\in \text{int}\, C_0:f(z)=2\pi\}$ contains three points, counted according to multiplicity. By~(4.8) and~(4.9), and the fact that $f(z)=g(z)+\eps/z$, these three points are close to $re^{2\pi ik/3},\,k=0,1,2$, where $r=\sqrt[3]{\eps/\pi}$. Each of these three pre-images of $2\pi$ must lie in $U_0$, since 
\[f(\Delta_0\cup\Delta'_0)\subset B(2\pi,6\sqrt{\eps})\subset \Delta_1\subset U_1,\quad\text{where } \Delta'_0=\{z:2\sqrt{\eps}\le|z|\le\sqrt[3]{\eps}\},\]
as can easily be checked using~(4.6),~(4.8) and~(4.9). Note that $\sqrt[3]{\eps}>2\sqrt{\eps}$, since $0<\eps<1/144$. Hence $k_0=3$, by Lemma~\ref{imagecircle}(c). By Lemma~\ref{imagecircle}(b), with $n=0$,  and the argument principle, the set $\{z\in \text{int}\, C_0:f'(z)=0\}$ contains three points, counted according to multiplicity, so $N_0\le 3$. Also, $c(U_1)=1$, so
\[c(U_0)=2+3(-1)+N_0\le 2,\]
by~(4.11). Since $U_0$ is multiply connected, we deduce that $c(U_0)=2$, as required.
\end{proof}

Our next example shows that Theorem~\ref{bddU}(b) is false for an unbounded Fatou component, even for $f\in M_F$. Here we use the approximation technique introduced by Eremenko and Lyubich~[\ref{EL}]. 
\begin{example}
There exists a function $f\in M_F$ with a bounded {\sconn} {\wand} $U$ such that \begin{itemize}
\item[(a)]
$f(U)$ is an unbounded {\sconn} component of $F(f)$ and $\partial f(U)$ consists of two unbounded components;
\item[(b)] 
$f^2(U)$ is a bounded {\dconn} component of $F(f)$;
\item[(c)]
$f^n(U),\, n\ge 3$, are bounded {\sconn} components of $F(f)$.
\end{itemize}
Thus $U_1=f(U)$ is unbounded and $c(U_1)=1<2=c(U_2)$.
\end{example} 

\begin{proof} 
Throughout this construction the parameters $\lambda$, $a$ and $\eps$ are the same as in Example~2, as are the sets $\Delta_n,\,n\ge 0$. In particular, $0<\eps<1/144$. We then define 
\[g_1(z)=z+\lambda \sin(z+a),\quad g_2(z)=4e^{z}-\eps/z\quad\text{and}\quad g_3(z)=0.\] 
Note that $g_1$ is the function called $g$ in Example~2. Also, let
\[E_1=\{z:\Re(z)\ge -0.6\},\quad E_2=\{z:\Re(z)\le -1.4\}\quad\text{and}\quad E_3=\{z:|z+1|\le 0.2\}.\]
It follows from Arakelyan's theorem [\ref{dG85}] that, for any $\delta>0$, there exists a transcendental entire function $g$ such that
\begin{equation}
|g(z)-g_k(z)|<\delta/2,\quad\text{for }z\in E_k, \, k=1,2,3,
\end{equation}
and $g$ is symmetric with respect to the real axis. The following lemma then completes the proof of Example~4.
\end{proof}
  
\begin{lemma}\label{approximation}
We can choose $\delta>0$ such that if $g$ is constructed as above, then the transcendental meromorphic function 
\begin{equation}
f(z)=g(z)+\frac{\eps}{z}+\frac{\delta/5}{z+1}
\end{equation} 
has the following properties.
\begin{itemize}
\item[(a)]
$F(f)$ has a sequence of components $V_n,\,n\ge 0$, with similar properties to the components $U_n$ in Example~\ref{ex2} (and Lemma~\ref{imagecircle}); in particular, $V_0$ is doubly connected, $V_n,\,n\ge 1$, are simply connected, and
\[\Delta_n\subset V_n\subset \{z:|z-2n\pi|<0.5\},\quad\text{for }n\ge 0.\]
\item[(b)]
$F(f)$ has an unbounded {\sconn} component $U'$ whose boundary $\partial U'$ consists of two unbounded components, such that $f(U')=V_0$.
\item[(c)]
$F(f)$ has a bounded {\sconn} component $U$ such that $f(U)=U'$.
\end{itemize}
\end{lemma}
\begin{proof}
Let $f_1(z)=g_1(z)+\eps/z$, so $f_1$ is the function called $f$ in Example~2. The proof of Example~2 depended on several properties of $f_1$. Part~(a) of Lemma~\ref{approximation} will follow if we show that these properties are also true for the function $f$ in this example.

First, $f_1$ is symmetric in the real axis and belongs to $M_F\cap M$, properties which are also true for the function $f$ defined by~(4.13).

Next, the proof of Example~2 depended on a finite number of statements, such as~(4.5) and Lemma~\ref{imagecircle}, all involving values of $z$ in $E_1$ and various small positive constants such as~$r_1$, which are true for the function $f_1$ and which remain true for the function~$f$ if we choose~$\delta>0$ small enough; for example, we have 
\[|f(z)-f_1(z)|=\left|g(z)-g_1(z)+\frac{\delta/5}{z+1}\right|< \delta,\qfor z\in E_1,\]
so~(4.5) is true for $f$ if $\delta>0$ is small enough, and 
\[|f'(z)-f'_1(z)|\le 10\delta, \qfor \Re(z)\ge -0.5,\] 
by Cauchy's estimate. Thus the statement~(4.10) in the proof of Lemma~\ref{imagecircle} is also true for~$f$ if~$\delta>0$ is small enough. 
 
To prove part~(b), we show that a certain component $U'$ of the pre-image of $V_0$ under $f$ is an unbounded {\sconn} component of $F(f)$. First, recall that \[\Delta_0=\{z:\sqrt{\eps}/2<|z|<2\sqrt{\eps}\}.\] 
It follows from~(4.12) and~(4.13) that if $\delta>0$ is small enough, then there exists $\rho>0$, depending on $\eps$ but not on $\delta$, such that $V_0$ surrounds $\{z:|z|\le\rho\}$. In particular, $\rho\le\sqrt{\eps}/2$. Then we take $C$ such that $8e^{-C}<\rho$, put
\[S=\{z:-C<\Re(z)<-2\},\]
and further require that $0<\delta<2e^{-C}$.

Let $\phi(z)=f(z)-4e^z$. Then, by~(4.12) and~(4.13), we have
\[|\phi(z)|=\left|g(z)-g_2(z)+\frac{\delta/5}{z+1}\right|<\delta,\qfor z\in E_2,\]
and hence
\[|\phi'(z)|<\frac{\delta}{0.6}<2\delta,\qfor z\in S,\]
by Cauchy's estimate. Now, 
\[|f(z)|\ge |4e^z|-|\phi(z)|>4e^{-C}-\delta>2e^{-C},\qfor z\in S,\]
so any path in $S$ which tends to $\infty$ is mapped by $f$ to a path which winds infinitely often round $\{z:|z|\le 2e^{-C}\}$. Hence $f$ has no finite asymptotic values in $S$. Also, since $0<\delta<2e^{-C}<\rho/4\le\sqrt{\eps}/8<1/96$, we have
\[|f(z)|>4e^{-2}-\delta>0.5,\qfor \Re(z)=-2,\]
\[0<4e^{-C}-\delta<|f(z)|<4e^{-C}+\delta<\rho,\qfor \Re(z)=-C,\]
and 
\[|f'(z)|=|4e^z+\phi'(z)|\ge 4e^{-C}-2\delta>0,\qfor z\in S.\]
It follows that $f:S\to f(S)$ is a covering map and $\partial f(S)$ lies outside $V_0$, by part~(a). Also, since $0<\delta<\sqrt{\eps}/8$, the vertical line $\{z:\Re(z)=\ln(\sqrt{\eps}/4)\}$ in $S$ is mapped by $f$ to a path in $\Delta_0\subset V_0$, which winds infinitely often round~0. Thus $f^{-1}(V_0)$ has a component $U'$ which is an unbounded {\sconn} domain contained in $S$, bounded by two unbounded continua in $S$ which are components of the pre-images under $f$ of the inner and outer components of $\partial V_0$. Thus $U'$ is a Fatou component of $f$ and $f(U')=V_0$, by Lemma~\ref{Herring}(a).
 
Now we show that $f$ is univalent on the punctured disc $D=\{z:0<|z+1|<\sqrt{\delta}/2\}$, which is contained in $E_3=\{z:|z+1|\le 0.2\}$. Put $h(z)=g(z)+\eps/z$. Then, by~(4.12) and~(4.13),
\[|h(z)|\le\frac{\delta}{2}+\frac{\eps}{0.8}<\frac{1}{50},\qfor z\in E_3,\]
since $0<\eps<1/144$ and $0<\delta<1/96$. Thus, by Cauchy's estimate,
\[|h'(z)|\le \frac{1}{50\,(0.2-\sqrt{\delta}/2)}<1/5,\qfor z\in \overline{D}.\]
Now suppose that $f(z_1)=f(z_2)$, where $z_1,z_2\in D$. Then
\[\left|\frac{\delta/5}{z_1+1}-\frac{\delta/5}{z_2+1}\right|
=|h(z_1)-h(z_2)|\le \frac{1}{5}|z_1-z_2|,\]
so $\delta\le |z_1+1|\,|z_2+1|\le(\sqrt{\delta}/2)^2$, which is false. Hence $f$ is one-one on $D$.

Also, for $z\in \partial D\setminus\{-1\}$, we have
\[|f(z)|=\left|h(z)+\frac{\delta/5}{z+1}\right|\le |h(z)|+\frac{\delta/5}{|z+1|}\le\frac{\delta}{2}+\frac{\eps}{0.8}+\frac{2\sqrt{\delta}}{5}
\le \sqrt{\eps},\]
provided that we also have $0<\delta<\eps$. For such $\delta$, the function $f$ maps $D$ univalently onto a domain which contains $\{z:|z|>\sqrt{\eps}\}$ and hence contains the component $U'$, since $\{z:|z|=\sqrt{\eps}\}\subset V_0$. Therefore $f^{-1}(U')$ has a bounded {\sconn} component~$U$ in $D$, which is a component of $F(f)$ such that $f(U)=U'$ and $-1\in \overline{U}$. This completes the proof of Lemma~\ref{approximation}.
\end{proof}

Our next example shows that Theorem~2(c) is also false for an unbounded Fatou component, even for $f\in M_F$.
\begin{example}
There exists a function $f\in M_F$ with a bounded {\iconn} {\wand} $U$ such that \begin{itemize}
\item[(a)]
$f(U)$ is an unbounded {\iconn} component of $F(f)$;
\item[(b)] 
$f^2(U)$ is contained in a bounded {\dconn} component of $F(f)$;
\item[(c)]
$f^n(U),\, n\ge 3$, are contained in bounded {\sconn} components of $F(f)$.
\end{itemize}
Thus $U_1=f(U)$ is unbounded and {\iconn}, and the eventual connectivity of~$U_1$ is~1.
\end{example}

\begin{proof}
The proof is similar to that of Example~3, but we replace the function $g_2$ used in that proof by
\[g_2(z)=e^{z}-\sqrt{\eps}-\frac{\eps}{z},\]
and then define $g$ and $f$, as before, to be symmetric in the real axis and satisfy~(4.12) and~(4.13). Recall that $\Delta_0=\{z:\sqrt{\eps}/2<|z|<2\sqrt{\eps}\}$, so $-\sqrt{\eps}\in\Delta_0$, and also that $0<\eps<1/144$.

As in Lemma~\ref{approximation}(a), we can take $\delta>0$ so small in~(4.12) and~(4.13) that $F(f)$ has a sequence of components $V_n,\,n\ge 0$, with similar properties to the components $U_n$ in Example~\ref{ex2} (and Lemma~\ref{imagecircle}); in particular, $V_0$ is doubly connected, $V_n,\,n\ge 1$, are simply connected, and
\begin{equation}
\Delta_n\subset V_n\subset \{z:|z-2n\pi|<0.5\},\quad\text{for }n\ge 0.
\end{equation}

Now, we introduce the connected compact set 
\[K=\{z:|z|=3\sqrt{\eps}/2\}\cup[-3\sqrt{\eps}/2,-5\sqrt{\eps}/4]\cup
\{z:|z+\sqrt{\eps}|=\sqrt{\eps}/4\},\]
which is a subset of $\Delta_0$, and put
\[L=\exp^{-1}(K+\sqrt{\eps}).\]
Then $L$ is an unbounded `vertical ladder' (the left edge straight and the right edge wavy), which has infinitely many horizontal rungs and is invariant under translation by $2\pi i$. We have $L\subset E_2$, since  $\ln(5\sqrt{\eps}/2)<-1.4$.
By~(4.12) and~(4.13), we have
\begin{equation}
|f(z)-e^z+\sqrt{\eps}|=\left|g(z)-g_2(z)+\frac{\delta/5}{z+1}\right|<\delta, 
\quad\text{for }z\in E_2,
\end{equation}
so 
\begin{equation}
f(z)\in \Delta_0\subset V_0,\quad\text{for }z\in L,
\end{equation}
provided that $0<\delta<\sqrt{\eps}/4$. Thus the set $L$ must lie in an unbounded component $U'$ of $F(f)$ such that $f(U')\subset V_0$. Now, the inner boundary component,~$\alpha_0$ say, of the doubly connected component $V_0$ is surrounded by~$\Delta_0$. Thus~(4.15) and~(4.16) imply that the image under~$f$ of the boundary of each hole of the ladder~$L$ must wind once round~$\alpha_0$. Hence, by the argument principle, each of the holes of~$L$ must contain a pre-image of~$\alpha_0$ under~$f$, so the component $U'$ is infinitely connected.

To complete the proof, we again use the fact that, for small enough $\delta>0$, the function $f$ maps the punctured disc $D=\{z:0<|z+1|<\sqrt{\delta}/2\}$ univalently onto a domain which contains $\{z:|z|>\sqrt{\eps}\}$.
\end{proof}

Our final example shows that Theorem~\ref{component} is false if $U$ is unbounded. See [\ref{Q94}, Theorem~1] and [\ref{BD99}, Theorem~6.1] for related properties of the Julia set of this function.

\begin{example}
The function $f(z)=ze^z$ has an unbounded immediate parabolic basin $U$ whose boundary $\partial U$ has components $\alpha$ and $\alpha'$ such that $f(\alpha)=\alpha'\setminus\{0\}$. 
\end{example}
\begin{proof}
The function $f$ has a parabolic fixed point at~0, with an associated immediate parabolic basin~$U$ that contains $(-\infty,0)$. The only singular values of $f$ are the finite asymptotic value~0 and the critical value $f(-1)=-1/e$.

Let $\Omega=\{z:\Re (z)\le 0, |\Im(z)|\le \pi/2\}$ and let $\Gamma^{\pm}$ be the parts of $\partial\Omega$ in the upper and lower open half-planes. Simple estimates show that 
\[
 f(\Omega\setminus\{0\})\subset {\rm int}\,\Omega,
\]
so $\Omega\setminus\{0\}\subset U$. Then take $G=\C\setminus \Omega$. Let $g$ be the branch of $f^{-1}$ such that $g(0)=0$, defined on a neighbourhood of~0, and analytically continue $g$ to $\C\setminus (-\infty,0]$ by using the monodromy theorem. Then $g(G)\supset (0,\infty)$, but
\[g(G)\cap \partial G=\emptyset,\quad\text{since}\quad 
f(\partial G\setminus\{0\})\subset \Omega.\]
Thus $g(G)\subset G$, so $\overline{g^n(G)}$, $n=1,2,\ldots\,$,  forms a decreasing sequence of continua in~$\hat{\C}$ with intersection $\Delta$, say, containing $[0,\infty)$. Then $\Delta\setminus\{{\infty}\}$ is completely invariant under $g$. 

Now let $S=\{z:\Re(z)\ge 0, |\Im(z)|\le \pi\}$ and $H=\{z:\Im(z)>0\}$. By considering the effect of $f$ on each of the half-lines
\[\{x+iy:x\ge 0\},\quad 0\le y\le\pi,\]
we see that $f$ maps the interior of $S\cap H$ univalently onto a {\sconn} domain which contains $G\cap H$.
Thus $g(G)\subset S$ and hence $\Delta\setminus\{{\infty}\}\subset S$. We can then deduce that $\Delta\setminus\{{\infty}\}=[0,\infty)$ by considering a point of $\Delta$ with maximal argument, and using the fact that $\arg f(z)=\arg z+y$, for $z\in S$.

We have $(0,\frac12\pi i)\subset U\cap \partial(S\cap H)$ and $f((0,\frac12\pi i))\subset {\rm int}\,\Omega\cap H\subset U\cap H$. Thus $g({\rm int}\,\Omega)\cap {\rm int}\,\Omega\ne\emptyset$, so both $g({\rm int}\,\Omega)$ and $g(\Gamma^+)$ are subsets of $U$, and the same therefore holds for $g^n(\Gamma^+)$, for all $n\ge 0$. Since $[0,\infty)$ does not meet $U$ and the curves $g^n(\Gamma^+)$ tend to $[0,\infty)$, we deduce that $\alpha'=[0,\infty)$ is contained in $\partial U$ and moreover forms a component of~$\partial U$.
 
Next let $h$ denote the branch of $f^{-1}$ which maps the interval $[-1/e,0)$ to $(-\infty,-1]$. We can analytically continue $h$ to $H$, and from $H$ across the three intervals of $\R\setminus \{0,-1/e\}$. Therefore the image of $H$ under $h$ is a domain bounded by three curves
\[h((-\infty,-1/e)),\quad h([-1/e,0))=(-\infty,-1],\quad h((0,\infty)),\]
each of which is a solution curve of the equation $\Im(ze^z)=0$. In particular, the curve $\alpha=h((0,\infty))$ is a complete branch of the graph $x=-y \cot y$. 

Now $\alpha\subset J(f)$, since $(0,\infty)\subset J(f)$. Also, $h(U\cap H)\subset U$, so $\alpha\subset \partial U$, since $(0,\infty)\subset \partial U$. Moreover $\alpha$ is a component of $\partial U$ since it is a maximal connected subset of $f^{-1}([0,\infty))$. However, $f(\alpha)=(0,\infty)=\alpha'\setminus\{0\}$ is not a component of $\partial U$, so the proof is complete.
\end{proof}

\section{References}
\begin{enumerate}

\item\label{iB75} I.N. Baker, The domains of normality of an entire function,
{\it Ann. Acad. Sci. Fenn. Ser. A I Math.\,}, 1 (1975), 277--283.

\item \label{iB76} I.N. Baker, An entire function which has wandering domains, {\it J. Austral. Math. Soc. Ser. A}, 22 (1976), 173--176.


\item \label{iB85} I.N. Baker, Some entire functions with multiply-connected wandering domains, {\it Ergodic Theory Dynam. Systems}, 5 (1985), 163--169.

\item\label{iB87} I.N. Baker, Wandering domains for maps of the punctured plane,
{\it Ann. Acad. Sci. Fenn. Ser. A I Math.\,}, 12 (1987), 191--198.

\item\label{BD99} I.N. Baker and P. Dom{\' i}nguez, Boundaries of unbounded Fatou components of entire functions, {\it Ann. Acad. Sci. Fenn. Math.\,}, 24 (1999), 437--464.

\item \label{BKL1} I.N. Baker, J. Kotus and L{\" u} Yinian, Iterates of meromorphic functions II: Examples of {\wand}s, {\it J. London Math. Soc.}, 42 (1990), 267--278.

\item \label{BKL2} I.N. Baker, J. Kotus and L{\" u} Yinian, Iterates of meromorphic functions I, {\it Ergodic Theory Dynam. Systems}, 11 (1991), 241--248.

\item \label{BKL3} I.N. Baker, J. Kotus and L{\" u} Yinian, Iterates of meromorphic functions III: Preperiodic domains, {\it Ergodic Theory Dynam. Systems}, 11 (1991), 603--618.

\item \label{wB93} W. Bergweiler, Iteration of meromorphic functions, {\it Bull. Amer. Math. Soc.}, 29 (1993), 151--188.

\item \label{wB95} W. Bergweiler, Invariant domains and singularities,
{\it Math. Proc. Camb. Phil. Soc.\,}, 117 (1995), 525--532.



\item \label{aB99} A. Bolsch, Periodic Fatou components of meromorphic functions, {\it Bull. London. Math. Soc.}, 31 (1999), 543--555.


\item \label{rD89} R.L. Devaney, Dynamics of entire maps, {\it Dynamical systems and ergodic theory}, Banach Center Publications 23, Polish Scientific Publishers, 1989, 221--228.


\item \label{Pat2} P. Dom\'{\i}nguez, Dynamics of transcendental meromorphic
functions, {\it Ann. Acad. Sci. Fenn. Math. Ser. A} (1), 23 (1998), 225--250.

\item \label{EL} A. Eremenko and M. Lyubich, Examples of entire functions with pathological dynamics, {\it J. London Math. Soc.} (2), 36 (1987), 458--468.

\item\label{dG85} D. Gaier, {\em Lectures on complex approximation}, Birkh{\"a}user, 1985.


\item \label{mH84} M. Herman, Exemples de fractions rationelles ayant une orbit dense sur la sph{\` e}re de Riemann, {\it Bull. Soc. Math. France}, 112 (1984), 93--142.

\item \label{mH98} M.E. Herring, Mapping properties of Fatou components, {\it Ann. Acad. Sci. Fenn. Math.}, 23 (1998), 263--274.

\item\label{K98} M. Kisaka, On the connectivity of Julia sets of {\tef}s, {\it Ergodic Theory
Dynam. Systems}, 18 (1998), 189--205.

\item\label{KS06} M. Kisaka and M. Shishikura, On multiply connected wandering domains of entire functions, To appear in {\em Transcendental dynamics and complex analysis}, Cambridge University Press, 2007.

\item \label{cMwR91} C. Mueller and W. Rudin, Proper holomorphic maps of plane regions,
{\it Complex Variables}, 17 (1991), 113--121.


\item \label{Nev} R. Nevanlinna, {\it Analytic functions}, Springer, 1970.

\item \label{New} M.H.A. Newman, {\it Elements of the topology of plane sets of points},
Cambridge University Press, 1961.

\item \label{kN60} K. Noshiro, {\it Cluster sets}, Springer-Verlag, 1960.

\item\label{Q94} J.Y. Qiao, The Julia set of the mapping $z \to z\exp(z+\mu)$,
{\it Chinese Science Bulletin\,}, 139 (1994), 529--533.

\item\label{QW06} L. Qiu and S. Wu, Weakly repelling fixed points and multiply-connected wandering domains of meromorphic functions, {\it Sci. China Ser. A}, 49 (2006), no.~2, 267--276.


\item\label{lR04} L. Rempe, On a problem of Herman, Baker and Rippon concerning Siegel disks, {\em Bull. London Math. Soc.\,}, 36 (2004), 516--518.

\item\label{pR94} P.J. Rippon, On the boundaries of certain Siegel discs, {\em C.R. Acad. Sci. Paris, S{\' e}rie I\,}, 319 (1994), 821--826.

\item\label{RS00} P.J. Rippon and G.M. Stallard, On sets where iterates of a
meromorphic function zip towards infinity, {\it Bull. London Math. Soc.},
32 (2000), 528--536.


\item\label{RS05} P.J. Rippon and G.M. Stallard, Escaping points of meromorphic functions with a finite number of poles, {\it Journal d'Analyse Math.\,}, 96 (2005), 225--245.


\item \label{gS91} G.M. Stallard, A class of meromorphic functions with no {\wand}s, {\it Ann. Acad. Sci. Fenn. Ser. A I Math.\,}, 16 (1991), 211--226.

\item\label{nS93} N. Steinmetz, {\em Rational iteration}, de Gruyter, 1993.



\end{enumerate}

\end{document}